\documentclass[11pt,reqno,fleqn]{amsart}

\usepackage{amsfonts,amsmath,amsthm,amssymb}
\usepackage{pdfsync}

\newtheorem{thm}{Theorem}[section]

\newtheorem{lemma}{Lemma}[section]
\newtheorem{prop}{Proposition}[section]

\newcommand{\R}{\ensuremath{\mathbb{R}}}
\newcommand{\Z}{\ensuremath{\mathbb{Z}}}
\newcommand{\C}{\ensuremath{\mathbb{C}}}
\newcommand{\N}{\ensuremath{\mathbb{N}}}

\newcommand{\F}{\ensuremath{\mathcal{F}}}

\newcommand{\MM}{\ensuremath{\mathcal{M}}}
\newcommand{\RR}{\ensuremath{\mathcal{R}}}
\newcommand{\LL}{\ensuremath{\mathcal{L}}}

\newcommand{\PP}{\ensuremath{\mathbf{P}}}

\newcommand{\Si}{\ensuremath{\mathfrak{S}}}

\newcommand{\1}{\ensuremath{\mathbf{1}}}

\newcommand{\fa}{\ensuremath{f(\frac{\log\,d_1}{\log\,R})}}
\newcommand{\fb}{\ensuremath{f(\frac{\log\,d_2}{\log\,R})}}

\newcommand{\x}{\ensuremath{\mathbf{x}}}
\newcommand{\y}{\ensuremath{\mathbf{y}}}
\newcommand{\z}{\ensuremath{\mathbf{z}}}
\newcommand{\s}{\ensuremath{\mathbf{s}}}
\newcommand{\bt}{\ensuremath{\mathbf{t}}}
\newcommand{\bal}{\ensuremath{\mathbf{\alpha}}}

\newcommand{\baa}{\ensuremath{\mathbf{a}}}
\newcommand{\bb}{\ensuremath{\mathbf{b}}}
\newcommand{\uu}{\ensuremath{\mathbf{u}}}
\newcommand{\vv}{\ensuremath{\mathbf{v}}}

\newcommand{\h}{\ensuremath{\mathbf{h}}}

\newcommand{\ls}{\ensuremath{\lesssim}}

\newcommand{\subs}{\ensuremath{\subseteq}}

\newcommand{\eps}{\ensuremath{\varepsilon}}

\newcommand{\La}{\ensuremath{\Lambda}}
\newcommand{\al}{\ensuremath{\alpha}}
\newcommand{\be}{\ensuremath{\beta}}
\newcommand{\ga}{\ensuremath{\gamma}}

\newcommand{\te}{\ensuremath{\theta}}
\newcommand{\de}{\ensuremath{\delta}}
\newcommand{\De}{\ensuremath{\Delta}}
\newcommand{\si}{\ensuremath{\sigma}}
\newcommand{\om}{\ensuremath{\omega}}

\newcommand{\hf}{\ensuremath{\widehat{f}}}

\newcommand{\eq}{\begin{equation}
\newcommand{\ee}{\end{equation}}}

\newcommand{\p}{\ensuremath{\mathfrak{p}}}

\numberwithin{equation}{section}

\begin{document}
\title{Almost prime solutions to diophantine systems of  high rank}
\author{\'Akos Magyar and Tatchai Titichetrakun}
\thanks{
The first author is supported in part by Grants DMS-1600840 and ERC-AdG. 321104.}

\begin{abstract}
Let $\F$ be a family of $r$ integral forms of degree $k\geq 2$ and $\LL=(l_1,\ldots,l_m)$ be a family of pairwise linearly independent linear forms in $n$ variables $\x=(x_1,...,x_n)$. We study the number of solutions $\x\in[1,N]^n$ to the diophantine system $\F(\x)=\vv$ under the restriction that $l_i(\x)$ has a bounded number of prime factors for each $1\leq i\leq m$. We show that the system $\F$ have the expected number of such ``almost prime" solutions under similar conditions as was established for existence of integer solutions by Birch.
\end{abstract}

\maketitle

\section{Introduction.} In general, proving the existence of integer solutions to a diophantine system is not possible, however the problem becomes much more feasible if the system has sufficiently large rank, or have high degree of symmetry. Indeed, it was shown by Birch \cite{Bi} and later by Schmidt \cite{Sch}, that local to global type asymptotic formulas can be attained for the number of solutions of diophantine systems whose rank is exponentially large with respect to its degree.
\\\\
It is natural expect that similar results should hold when the solutions are restricted to special sequences such as the prime numbers. For diagonal systems this has been done by Hua \cite{Hua}, who, extending the methods of Vinogradov \cite{Vin}, derived asymptotics for the number of solutions consisting of primes. For diagonal quadratic and cubic equations various further refinements were obtained \cite{Wol}, \cite{Liu}, \cite{Bru} considering both prime and almost prime solutions combining the circle and sieve methods.
\\\\
It was shown only recently \cite{CM1} that local to global type principles hold for the number of prime solutions for any diophantine system whose rank is sufficiently large with respect to the degree and number of equations. The methods there employed certain ideas from arithmetic combinatorics, leading to tower-exponential type conditions on the rank. It is expected that such results should hold under the essentially the same rank conditions as needed for the existence of integer solutions. For linear systems a crucial step in studying prime solutions is to establish certain pseudo-randomness properties of the \emph{almost primes}, expressed in the so-called \emph{linear forms conditions} of Green and Tao \cite{GT}. The aim of the present article is to obtain similar results when the variables are restricted to the solution set of a diophantine system, see Theorem 1.2 below. As corollary
we show that if the system $\F$ has sufficiently large rank, then there are solutions $\x$ to $\F(\x)=\vv$ such that $l_i(\x)$ has a bounded number of prime factors, simultaneously for $1\leq i\leq m$. In fact, subject to some natural local conditions, one obtains the expected number of such solutions $x\in [1,N]^n$, as $N\to\infty$, see Theorem 1.1.

In the special case $m=n$, $l_i(\x)=x_i$ one gets almost prime solutions, a fact which is well-known in sieve theory, as long as one is able count the number of solutions in congruence classes of large moduli \cite{HR}. Our method is based on the sieve of Goldston-Yildirim \cite{GMPY} and it may extend to study almost prime values of certain systems of higher degree forms $\mathcal{G}=(G_1,\ldots,G_m)$ at least when the forms $G_i$ has ``linear parts" \cite{TZ}, although we do not pursue this here. Such results might also serve as a step toward a different approach to count prime solutions for translation invariant diophantine systems, considering the primes as a dense subset of the almost primes. Solutions to translation invariant systems in dense subsets of integers have been established recently in \cite{Hen}, \cite{Keil}, \cite{Matt}.
\\\\
To state our main results, following \cite{Bi}, define the rank of the system $Rank (\F)$, as the codimension of the singular variety $V^*_\F\subs \C^n$, consisting of points $\z\in\C^n$ where the Jacobian $\partial\F/\partial\z$ drops rank.
Note that for a single quadratic form $F(\x)=Ax\cdot x$ this agrees with the rank of the underlying matrix $A$. Recall the classical result of Birch \cite{Bi} which states that if
\[ Rank (\F)> r(r+1)(k-1)2^{k-1},\]
then the number of solutions $\x\in [1,N]^n$ to the system $\F(\x)=\vv$ satisfies
\[M_F(N)\,=\, N^{n-kr}\, J(N^{-k}\vv)\prod_p \si_p(\vv)\,+\,O(N^{n-kr-\de}),\]
for some $\de>0$. Here $\si_p(\vv)$ is the local factor representing the density of solutions among the $p$-adic integers and $J(\uu)$ is the so-called singular integral representing the density of real solutions $\x\in[0,1]^n$ to $\F(\x)=\uu$. It is well-known that $\si_p(\vv)\neq 0$ if there is a non-singular solution among the $p$-adic integers and $J(\uu)\neq 0$ if there is a non-singular real solution $\uu\in (0,1)^n$, see \cite{Bi},\cite{Sch}.
\\\\
For $0<\eps<1$ and $N\geq 1$ let $\PP^\eps(N)$ denote set of natural numbers such that each prime divisor of $x$ is at least $N^\eps$. Note that each $x\in \PP^\eps(N)$, $1\leq x \leq N$ has at most $l=[1/\eps]$ prime factors. For given $\vv\in\Z^n$ let
\[\mathcal{M}_{\F,\LL}^\eps(N):=|\{\x\in [N]^n;\ \F(\x)=\vv,\,l_i(\x)\in \PP^\eps(N)\ 1\leq i\leq m\}.\]
For a fixed prime $p$ define the local density, depending on the family $\LL$ as well,
\eq\label{0.1}
\si^*_p(\vv):=\,\frac{p^m}{(p-1)^m}\,\lim_{l\rightarrow \infty} p^{l(n-r)}\,|\{x\in\Z_{p^l}^n;\
\F(\x)\equiv\vv\ (mod\ p^l),\ (\LL(\x),p)=1\}|,
\ee
provided the limit exists, where by $(\LL(\x),p)=1$ we mean that $l_i(\x)$ is not divisible by $p$ for every $1\leq i\leq m$. Our main result is the following.

\begin{thm}\label{thm1.1}
Let $\F=(F_1,\ldots,F_r)$ be a system of $r$ integral forms of degree $k$ in $n$ variables such that
\eq\label{1.0} Rank (\F)> r(r+1)(k-1)2^{k-1}.\ee
Let $\LL=(l_1,\ldots,l_m)$ be a pairwise linearly independent family of integral linear forms.
Then there exists a constant $\eps=\eps(m,k,r)>0$ such that
\eq\label{1.1}\mathcal{M}_{\F,\LL}^{\eps}(N)\geq c_0\ N^{n-kr}\,(log\,N)^{-m} J(N^{-k}\vv)\,\prod_p \si^*_p(\vv),\ee
for some constant $c_0=c_0(\F,\LL)>0$. Moreover one may take $\eps= (2^8\,m^{3/2}r^2(r+1)(r+2)k(k+1))^{-1}$.
\end{thm}

Thus the conditions on the existence of almost prime solutions are essentially the same as those of for integer solutions, the difference being the natural requirement to have $p$-adic solutions at which the forms $l_i$ take values among the $p$-adic units.
\\\\
The key to prove Theorem \ref{thm1.1} is to study a weighted sum over the solutions with weights that are concentrated on numbers having few prime factors. Such weights have been defined by Goldston, Pintz and Yildirim \cite{GPY} in their seminal work on gaps between the primes. For given $0<\eta<1$ and  $1\leq R\leq N^\eta$ and define
\eq\label{1.1.1}\La_R(m):=\sum_{d|m} \mu(d)\,f(\frac{log\,d}{log\,R}),\ee
where $\mu$ is the M\"{o}bius function.
\\\\
We set $f(x)=(1-x)_+^{8m}$ and follow the Fourier analytic approach in \cite{TG} as opposed to the contour integration method of \cite{GPY}. Using the so-called "$W$-trick" introduced by Green and Tao in \cite{GT1} to bypass the contribution of small primes, let $\om=\om (\F,\LL)$ be a fixed positive integer depending only on the system $(\F,\LL)$, and let $W:=\prod_{p\leq \om} p$, the product of primes up to $\om$. We write $(\x,W)=1$ if $\x=(x_1,\ldots, x_n)$ such that $(x_i,W)=1$ for all $1\leq i\leq n$. Under the conditions of Theorem \ref{thm1.1} our key estimates will be

\begin{thm}\label{thm1.2} Let $\F=(F_1,\ldots,F_r)$ be a system of $r$ integral forms of degree $k$ in $n$ variables satisfying the rank condition \eqref{1.0} and let $\LL=(l_1,\ldots,l_m)$ be a pairwise linearly independent family of integral linear forms. Let $\La_R$ be as given in \eqref{1.1.1} associated to the function $f(x)=(1-x)_+^{8m}$. Then there exists $\eta=\eta(r,k)>0$ such that for $R\leq N^\eta$,

\begin{multline}\label{1.2.1}
\sum_{\substack{\x\in[N]^n,\,\F(\x)=\vv\\ (\LL(x),W)=1}} \La_R^2(l_1(\x)\cdots l_m(\x))\,=\,
\ c_m(f)\,N^{n-rk}(log\,R)^{-m} \Si^*(N,\vv)\\
\times\,(1+o_{\om\to\infty}(1)+O((log\,R)^{-1}))\quad\quad\quad
\end{multline}
\\
where $\ \Si^*(N,\vv)=J(N^{-k}\vv)\prod_p \si^*_p(v)$, and
\[
c_m(f)=\int_0^\infty f^{(m)}(x)^2\,\frac{x^{m-1}}{(m-1)!}\,dx.\]
Moreover one may take $\ \eta(r,k)=(8r^2(r+1)(r+2)k(k+1))^{-1}$.
\\\\
In addition, for given $0<\eps<\eta\leq \eta(r,k)$,

\begin{multline}\label{1.2.2}\sum_{\substack{\x\in[N]^n,\,\LL(\x)\notin (\PP^\eps(N))^m\\(\LL(\x),W)=1,\,\F(\x)=\vv}} \La_R^2(l_1(\x)\cdots l_m(\x)) \,\leq\,c'_{m+1}(f)\left(\frac{\eps}{\eta}\right)^2\ N^{n-rk}(log\,R)^{-m} \Si^*(N,\vv)\\
\times\,(1+o_{\om\to\infty}(1)+O((log\,R)^{-1}))\quad\quad\quad
\end{multline}
with
\[c'_{m+1}(f)=2m\,\int_0^\infty f^{(m+1)}(x)^2\,\frac{x^{m-1}}{(m-1)!}\,dx.\]
\end{thm}
$\ $\\
Note that the implicit constants in the above error terms may depend on the parameters $m,n,k,r$ attached to the system $(\F,\LL)$, however we're not indicating that as we're interested in the asymptotic as $N\to\infty$ for a fixed system $(\F,\LL)$; and only in the explicit form of the parameter $\eta=\eta(r,k)$ which determines the number of prime factors of the values $l_i(\x)$ taken at the solutions to $\F(\x)=\vv$.
\\\\
In the proof of Theorem \ref{thm1.2} we'll use the asymptotic for the number of integer solutions $\x\in[N]^n$ to $\F(\x)=\vv$
subject to the congruence condition $\x\equiv \s\ (mod\ D)$, where $D$ is a modulus bounded by a sufficiently small power $N$. This follows from the Birch-Davenport variant of the circle method described in \cite{Bi}, and is summarized in

\begin{prop}\label{prop1.1} Let $\F=(F_1,\ldots,F_r)$ be a family of integral forms of degree $k$ satisfying the rank condition \eqref{1.0}, and For given $D\in\N$ and $\s\in\Z^n$ let
\eq\label{1.3.1} \RR_N(D,\s;\vv):= |\{\x\in[N]^n;\ \x\equiv\s\ (mod\ D),\ \F(\x)=\vv\}|.\ee
\\
Then there exists a constant $\de'=\de'(k,r)>0$ such that the following holds.
\\\\
(i) If $0<\eta' \leq \eta'(r,k):=\frac{1}{4r^2(r+1)(r+2)k^2}$ then for every $1\leq D\leq N^\frac{\eta'}{1+\eta'}$ and $\s\in\Z^n$ one has the asymptotic

\eq\label{1.3.2}
\RR_N(D,\s;\vv)=N^{n-rk}D^{-n}\,J(N^{-k}\vv)\,\prod_p \si_p(D,\s,\vv)+\,O(N^{n-rk-\de'}D^{-n}).
\ee
\\
(ii) Moreover if
\eq\label{1.3.3}
Rank\,(\F) > (r(r+1)(k-1)+rk)2^k
\ee
then the asymptotic formula \eqref{1.3.2} holds for $\eta\leq \frac{1}{4r(r+2)k}$.
\end{prop}
Here $\si_p(D,\s,\vv)$ represents the density of the solutions among the $p$-adic numbers, more precisely
\eq\label{1.3.4}
\si_p(D,\s,\vv)=\lim_{l\to\infty} \si_p^l(D,\s,\vv),\ \ \si_p^l(D,\s,\vv)=p^{-l(n-r)}|\{\x\in\Z_{p^l}^n;\ \F(D\x+\s)\equiv \vv\ (mod\ p^l)\}|.
\ee
\\
The product form of the main term in \eqref{1.3.2} will allow as to write the expressions on the left side of \eqref{1.2.1} and \eqref{1.2.2} as an integral over an Euler product, which can be asymptotically evaluated using the sieve methods. To understand the local factors of The Euler product one needs to analyze the number of solutions to the system $\F(\x)=\vv$ mod $p$, this has been done it \cite{CM2} based on adapting  Birch's method to finite fields. Here the $W$-trick is quite useful as one has to consider only sufficiently large primes, for which the rank of the $mod\ p$ - reduced variety $V_\F=\{\F(\x)=\vv\}$ remains sufficiently large.
\\\\
The information needed about the Euler factors

\eq\label{1.4.1}
\ga_p(\vv):= \frac{p^{-n}}{\si_p(\vv)}\sum_{\substack{\s\in\Z_p^n,\\ \F(s)\equiv\vv\ (mod\ p)}} \1_{p|l_1(\s)\cdots l_m(\s)}\, \si_p(p,\s,\vv),\ee
is summarized in

\begin{prop}\label{prop1.2} Let $\F$ be a family of $r$ integral forms of degree $k$. If $rank(\F)>r(r+1)(k-1)2^k$ then for all sufficiently large primes $p>\om=\om(\F,\LL)$ one has
\eq\label{1.4.2}
\ga_p(\vv)=\frac{m}{p}+O(p^{-2}).\ee
\end{prop}

\bigskip

\subsection{Outline and Notations.} The facts about the number of solutions to diophantine systems among integers in a given residue class with respect to a small modulus will be given in Section 4 and will be used throughout the paper. The arguments are straightforward extensions of those of Birch discussed in \cite{Bi}.
In Section 3 we carry out the analysis of certain local factors attached the primes, and prove Proposition \ref{prop1.2}. Here we rely on certain results obtained in \cite{CM2} on the number of solutions of diophantine systems over finite fields, and some well-known facts in algebraic geometry \cite{GL}, \cite{Shm} about the size and the stability of the dimension of homogeneous algebraic sets when reduced $mod\ p$ .
Somewhat unusually, we will prove our main results in Section 2, using the results of Section 3 and Section 4. This is to separate our main arguments relying on the sieve of Goldston-Pintz-Yildirim \cite{GPY}, from those to count integer solutions of diophantine systems based on the Hardy-Littlewood method of exponential sums.
\\\\
The symbols $\mathbb{Z}$, $\mathbb{Q}$, $\mathbb{R}$, and $\mathbb{C}$ denote the integers, the rational numbers, the real numbers, and the complex numbers, respectively. We write $\Z_N$ for the group  $\mathbb{Z}\slash N\Z$ as well as $Z_N^*$ for the multiplicative group reduced residue classes $(mod\ N)$. If $X$ is a set then $\mathbf{1}_X$ denotes a characteristic function for $X$ in a specified ambient space, and on occasion, the set $X$ is replaced by a conditional statement which defines it.
The Landau $o$ and $O$ notation is used throughout the work. The notation $f\ls g$ is sometimes used to replace $f=O(g)$, the implicit constants may depend on the fixed parameters $m,n,k$ and $r$ however we usually do not denote the dependence on them.  By $o_{\om\to\infty}(1)$ we denote a quantity that tends to 0 as $\om\to \infty$. All our estimates are asymptotic as $N\to\infty$, and always assume that $N$ is sufficiently large with respect to $\om$ and the parameters attached to the system $(\F,\LL)$.

\medskip

\section{Proof of the main results.} In this section we introduce the Euler product representation of the weighted sums over the solutions defined in \eqref{1.2.1} and \eqref{1.2.2}, and prove Theorems 1.1 and 1.2 using the main results of Section 3 and Section 4.
\\\\
Let $\phi_N$ be the indictor function of the cube $[1,N]^n$, $\mu$ the M\"{o}bius function and write $\sum'$ for sums restricted to square-free numbers. We'll also use the customary notations $[a,b]$ and $(a,b)$ for the least common multiple and greatest common factor of the numbers $a$ and $b$. If $\bb=(b_1,\ldots,b_n)\in\Z^n$ is such that $(b_i,W)=1$ for all $1\leq i\leq n$, then we write $(\bb,W)=1$. We write $a|b$ If $a$ divides $b$ and $\1_{a|b}$ for the indicator function of this relation. We start by making a few immediate observations about the local factors $\si_p(D,\s,\vv)$ defined in \eqref{1.3.4}.

\begin{lemma}\label{lem2.1} Let $D$, $W$ be square free numbers such that $(D,W)=1$ and let $p$ be a prime.\\
If $(p,DW)=1$ then
\eq\label{2.5.1}
\si_p(DW,\bt,\vv)=\si_p(\vv).
\ee
If $p|D$ and $\bt\equiv \s\ (mod\ D)$ then one has
\eq\label{2.5.2}
\si_p(DW,\bt,\vv)=\si_p(p,\bt,\vv)=\si_p(p,\s,\vv).
\ee
Similarly, if $p|W$ and $\bt\equiv \bb\ (mod\ W)$ then
\eq\label{2.5.3}
\si_p(DW,\bt,\vv)=\si_p(p,\bt,\vv)=\si_p(p,\bb,\vv).
\ee
\end{lemma}

\begin{proof} To see \eqref{2.5.1} note that for any $l\in\N$ the transformation $\x\to DW\x+\bt$ is one-one and onto on $\Z_{p^l}^n$.
If $p|D$ then $D=pD'$ with $(p,D'W)=1$, and one may write $DW\x+\bt=p(D'W\x)+\bt$ and the first equality in \eqref{2.5.2} follows by making a change of variables $\y:=D'W\x$ on $\Z_{p^l}^n$. Also $p\y+\bt=p(\y+\uu)+\s$ and the second equality follows by replacing $\y$ with $\y+\uu$. Interchanging the role of $D$ and $W$ \eqref{2.5.3} follows.
\end{proof}
$\ $\\
Let $(\F,\LL)$ be system of integral forms, as given Theorem \ref{thm1.1}. Let $R:=N^\eta$, with $\eta=\eta(r,k)$ as defined in Theorem \ref{thm1.2}, and let $W=\prod_{p\leq \om}p$ for a sufficiently large constant $\om=\om_\F$. To show the validity of \eqref{1.2.1} and \eqref {1.2.2} we define the sums depending on a parameter $\bb$, such that $(\LL(\bb),W)=1$,

\eq\label{2.4}
S_{W,\bb}(N):= \sum_{\substack{\x\equiv\bb\ (mod\ W)\\ \F(\x)=\vv}} \La_R^2(l_1(\x)\cdots l_m(\x))\,\phi_N(\x),
\ee
and for $\om<q$, $q$ prime
\eq\label{2.5}
S_{W,\bb}(N,q):= \sum_{\substack{\x\equiv\bb\ (mod\ W)\\ \F(\x)=\vv}} \1_{q|l_1(\x)\cdots l_m(\x)}\ \La_R^2(l_1(\x)\cdots l_m(\x))\,\phi_N(\x).
\ee

\begin{lemma}\label{lem2.2} We have
\eq\label{2.6}
S_{W,\bb}(N)=N^{n-dr} J(N^{-k}\vv) W^{-n} \Si_{W,\bb}(\vv)\ \sideset{}{'}\sum_{(D,W)=1} h_D(R)\,\ga_D(\vv)\ +O(N^{n-rk-\de}),
\ee
and similarly
\eq\label{2.7}
S_{W,\bb}(N,q)=N^{n-dr} J(N^{-k}\vv) W^{-n} \Si_{W,\bb}(\vv)\ \sideset{}{'}\sum_{(D,W)=1} h_D(R)\,\ga_{[D,q]}(\vv)\ +O(N^{n-rk-\de}),
\ee
where
\eq\label{2.8}
\Si_{W,\bb}(\vv):=\prod_{p|W}\si_p(p,\bb,\vv)\prod_{p\nmid W}\si_p(\vv),
\ee
and
\eq\label{2.9}
\ga_D(\vv):=D^{-n}\sum_{\substack{\s\in \Z_D^n\\ \F(\s)\equiv\vv\ (mod\ D)}} \1_{D|s_1\cdots s_n} \prod_{p|D} \frac{\si_p(p,\s,\vv)}{\si_p(\vv)},
\ee
\eq\label{2.10}
h_D(R):=\sum_{[d_1,d_2]=D} \mu(d_1)\mu(d_2)\, \fa\,\fb.
\ee
\end{lemma}

\begin{proof} We start by writing

\begin{eqnarray}\label{2.11}
S_{W,\bb}(N)&=& \sum_{\substack{x\equiv\ \bb\ (mod\ W)\\\F(\x)=\vv}} \phi_N(\x)\sideset{}{'}\sum_{\substack{[d_1,d_2]|l_1(\x)\cdots l_m(\x)\\ ([d_1,d_2],W)=1}} \mu(d_1)\mu(d_2)\, \fa \,\fb\nonumber\\
&=&\sideset{}{'}\sum_D\sum_{[d_1,d_2]=D} \mu(d_1)\mu(d_2)\, \fa\,\fb\ \sum_{\substack{\x\equiv\bb\ (mod\ W)\\ \F(\x)=\vv}} \1_{D|l_1(\x)\cdots l_m(\x)}\phi_N(\x)\nonumber\\
\end{eqnarray}
\\
The inner sum in $\x$ is zero unless $(D,W)=1$. Indeed if there is a prime $p$ such that $p|D$ and $p|W$, then $p|l_i(\x)$ for some $1\leq i\leq m$ and hence $l_i(\bb)\equiv l_i(\x)\equiv 0\ (mod\ p)$ contradicting our assumption $(l_i(\bb),W)=1$. The conditions $D|l_1(\x)\cdots l_m(\x)$ and $\x\equiv \bb\ (mod\ W)$ depend only on $\x\ (mod\ DW)$, thus one may write
\eq\label{2.11a}
\sum_{\substack{\x\equiv\bb\ (mod\ W)\\ \F(\x)=\vv}} \1_{D|l_1(\x)\cdots l_m(\x)}\phi_N(\x)=\sum_{\substack{\bt\in \Z_{DW}^n,\ \bt\equiv\bb\ (mod\ W)\\ \F(\bt)\equiv\vv\ (mod\ DW)}} \1_{D|l_1(\bt)\cdots l_m(\bt)}
\sum_{\substack{\x\equiv \bt\ (mod\ DW)\\\F(\x)=\vv}} \phi_N(\x).
\ee
Since $D\leq R^2\leq N^\frac{2\eta}{1+\eta}$; by Proposition \ref{prop1.1} this further equals to
\eq\label{2.11b}
\sum_{\substack{\bt\in \Z_{DW}^n\ \bt\equiv\bb\ (mod\ W)\\ \F(\bt)\equiv\vv\ (mod\ DW)}} \1_{D|l_1(\bt)\cdots l_m(\bt)}
N^{n-kr}(DW)^{-n}\,( J(N^{-k}\vv)\prod_p \si_p(DW,\bt,\vv)\\ +\ O(N^{-\de'})),
\ee
with a constant $\de'>0$ depending on the system $\F$, given in Proposition \ref{prop1.1}.
To estimate the contribution of the error terms to the sum $S_{W,\bb}(N)$ given in \eqref{2.11}, note that for a given $D$ the number of pairs $d_1,d_2$ for which $[d_1,d_2]=D$ is $\ls_\tau D^\tau$ for all $\tau>0$, and the summation in $D$ is restricted to $D\leq R^2$ as the function $f(x)$ is supported on $x\leq 1$. Also for a square-free integer $D$ it is easy to see that the number of $\s\in\Z_D^n$ such that $D|l_i(\s)$ is at most $D^{n-1}$. Thus the total error obtained in \eqref{2.11} is bounded by
\eq\label{2.12a}
E_{W,\bb}(N)\ls_\tau N^{n-kr-\de'}\,\sum_{D\leq R^2} D^{-1+\tau}\ls N^{n-kr-\de'/2}.
\ee
\\
The sum in \eqref{2.11b} can be evaluated by a routine calculation using the Chinese Remainder Theorem and the properties of the local factors $\si_p(DW,\bt,\vv)$, given in Lemma \ref{lem2.1}. Indeed, to every $\bt\in \Z_{DW}^n$ satisfying $\bt\equiv \bb\ (mod\ W)$ there is a unique $\s\in\Z_D^n$ such that $\bt\equiv \s\ (mod\ D)$, and in that case $\F(\bt)\equiv \vv\ (mod\ DW)$ is equivalent to $\F(\s)\equiv\vv\ (mod\ D)$ using our assumption that $\F(\bb)\equiv\vv\ (mod\ W)$. Thus
\\
\[
\sum_{\substack{\bt\in \Z_{DW}^n,\ \bt\equiv\bb\ (mod\ W)\\ \F(\bt)\equiv\vv\ (mod\ DW)}} \1_{D|l_1(\bt)\cdots l_m(\bt)} \prod_{p|W}\si_p(p,\bb,\vv)\prod_{p|D} \si_p(p,\bt,\vv) \prod_{p\nmid DW}\si_p(\vv)
\]
\eq\label{2.13}
= \prod_{p|W}\si_p(p,\bb,\vv)\prod_{p\nmid W}\si_p(\vv)\,\sum_{\substack{\s\in \Z_D^n\\ \F(\s)\equiv\vv\ (mod\ D)}} \1_{D|l_1(\s)\cdots l_m(\s)} \prod_{p|D} \frac{\si_p(p,\s,\vv)}{\si_p(\vv)}.
\ee
\\
Then \eqref{2.6} follows from \eqref{2.11}-\eqref{2.13}, and to see the validity of \eqref{2.7} it is enough to remark that carrying out the calculation in \eqref{2.11} for the sum $S_{W,\bb}(N,q)$, the only difference is that the indicator function $\1_{D|l_1(\x)\ldots l_m(\x)}$ is replaced by $\1_{[D,q]|l_1(\x)\ldots l_m(\x)}$.
\end{proof}
$\ $\\
The sum
\eq\label{2.12.5}
S_W(f,\ga):=\sideset{}{'}\sum_{(D,W)=1} \ga_D(\vv) h_D(R)
\ee
can be asymptotically evaluated by sieve methods, see \cite{GPY}, \cite{TG}; we will sketch the approach in \cite{TG} and indicate how to modify the argument to obtain an asymptotic for the related sum
\eq\label{2.12.6}
S_{W,q}(f,\ga):=\sideset{}{'}\sum_{(D,W)=1} \ga_{[D,q]}(\vv) h_D(R)
\ee
needed for the ``concentration" estimate \eqref{1.2.2}.

\begin{lemma}\label{lem2.3} (\cite{TG}, Prop. 10) Let $\ga_D(\vv)$ be a multiplicative function satisfying estimate \eqref{1.4.2}. Then one has
\eq\label{2.13.1}
S_W(f,\ga)=\,(\frac{\phi(W)}{W}\, log\,R\,)^{-m}\,\int_0^\infty (f^{(m)}(x))^2\,\frac{x^{m-1}}{(m-1)!}\,dx\ \times (1+\,o_{\om\to\infty}(1)+O_W((log\,R\,)^{-1})).\ee

Moreover if $q>\omega$ is a prime then
\begin{align}\label{2.14}
S_{W,q}(f,\ga) \,&=&\,\frac{m}{q}\,(\frac{\phi(W)}{W} log\,R\,)^{-m}\,\int_0^\infty (f^{(m)}(x)-f^{(m)}(x+\frac{log\, q}{log\, R}))^2\,\frac{x^{m-1}}{(m-1)!}\,dx\,\times\nonumber\\
&\ &\times\ (1+o_{\om\to\infty}(1)+\,O_W((log\,R\,)^{-1})).
\end{align}
\end{lemma}

\begin{proof} We specify $f(x)=(1-|x|)_+^{8m}$,  the function $e^x f(x)$ is compactly supported and $f^{(8m)}(x)=c_m (1-|x|)_+$ hence its Fourier transform, denoted by $\hf(t)$, satisfies $|\hf(t)|\ls (1+|t|)^{-(8m+2)}$. Substituting the Fourier inversion formula
\[e^x f(x)=\int_\R e^{-itx} \hf(t)\,dt\]
into \eqref{2.10} one obtains
\eq\label{2.14.1}
h_D(R)=\int_\R\int_\R \sum_{[d_1,d_2]=D} \mu(d_1)\mu_(d_2) d_1^{-\frac{1+it_1}{log\,R}} d_2^{-\frac{1+it_2}{log\,R}} \hf(t_1)\hf(t_2)\,dt_1 dt_2=:\int_\R\int_\R g_D(t_1,t_2)\,dt_1 dt_2.
\ee
The function $g_D(t_1,t_2)H_D(R)$ is multiplicative in $D$ hence
\[
\sideset{}{'}\sum_{(D,W)=1} g_D(t_1,t_2) \ga_D(\vv)=\prod_{p>\om} (1+g_p(t_1,t_2)) \ga_p(\vv),
\]
which gives

\eq\label{2.15}
S_{W}(f,\ga) = \int_\R\int_\R \prod_{p>\om} \left(1-\frac{\ga_p(\vv)}{p^\frac{1+it_1}{log\,R}}-\frac{\ga_p(\vv)}{p^\frac{1+it_2}{log\,R}}+\frac{\ga_p(\vv)}
{p^\frac{2+it_1+it_2}{log\,R}}\right)\ \hf(t_1)\hf(t_2)\,dt_1dt_2.
\ee
By Proposition \ref{prop1.2}
\[log\,\left|1-\frac{\ga_p(\vv)}{p^\frac{1+it_1}{log\,R}}-\frac{\ga_p(\vv)}{p^\frac{1+it_2}{log\,R}}+\frac{\ga_p(\vv)}
{p^\frac{2+it_1+it_2}{log\,R}}\right|\, \leq \, 3m\,p^{-1-\frac{1}{log\,R}}+O(p^{-2}).\]
By the well-know asymptotic
\[\sum_p p^{-1-\frac{1}{log\,R}} = log\,log\,R +O(1),\]
we see that the integrand in \eqref{2.15} is bounded by $C (log\,R)^{3m}(1+|t_1|)^{-8m-2} (1+|t_2|)^{-8m-2}$. Integrating over the range $|t_1, |t_2|> \sqrt{log\,R}$ gives

\[S_W(f,\ga)\,=\, \int_{|t_1|,|t_2|\leq \sqrt{log\,R}} \prod_{p>\om} (1-\frac{\ga_p(\vv)}{p^\frac{1+it_1}{log\,R}}-\frac{\ga_p(\vv)}{p^\frac{1+it_2}{log\,R}}+\frac{\ga_p(\vv)}
{p^\frac{2+it_1+it_2}{log\,R}})\ \hf(t_1)\hf(t_2)\,dt_1dt_2 + O(log^{-m-1} R).\]

Let, for $Re(s)>1$,
\[\zeta_W(s):=\prod_{p>\om} (1-\frac{1}{p^s})^{-1}=\zeta(s) \prod_{p\leq \om}(1-\frac{1}{p^s}).\]
From \eqref{1.4.2} it is easy to see that
\eq\label{2.20}
\prod_{p>\om} (1-\frac{\ga_p(\vv)}{p^\frac{1+it_1}{log\,R}}-\frac{\ga_p(\vv)}{p^\frac{1+it_2}{log\,R}}+\frac{\ga_p(\vv)}
{p^\frac{2+it_1+it_2}{log\,R}})\,=\,\frac{\zeta_W^{m}(1+s_1+s_2)}{\zeta_W^{m}(1+s_1)\zeta_W^{m}(1+s_2)}\,
(1+o_{\om\to\infty}(1)),
\ee
with $s_1=1+\frac{1+it_1}{log\,R}$, $s_2=1+\frac{1+it_2}{log\,R}$. On the range $|t_1|, |t_2|\leq \sqrt{log\,R}$ one has that
\[\prod_{p\leq\om} (1-p^{-s})=\prod_{p\leq\om} (1-p^{-1})(1+o_{\om\to\infty}(1))=\frac{\phi(W)}{W} (1+o_{\om\to\infty}(1)).\]
Thus from the basic property $\zeta(s)=(s-1)^{-1}+O(1)$ for $s$ near 1, it follows
\[\zeta_W(s)=\frac{1}{s-1}\frac{\phi(W)}{W}\,(1+o_{\om\to\infty}(1)).\]
Substituting this into \eqref{2.20} gives

\begin{multline*} S_W(f,\ga)\,=\,(\frac{\phi(W)}{W}\,log\,R)^{-m} \int_{|t_1|,|t_2|\leq \sqrt{log\,R}} \frac{(1+it_1)^m (1+it_2)^m}
{(2+it_1+it_2)^m}\,\hf(t_1)\hf(t_2)\,dt_1dt_2\,\times\\ \times\,(1+o_{\om\to\infty}(1)).
\end{multline*}
Note that by the quick decrease of $\hf(t)$ the integration in $t_1$ and $t_2$ can be extended to $\R$ by making an error of $O((log\,R)^{-3m-1})$.
Finally, using the identities
\eq\label{2.21}(2+it_1+it_2)^{-m}=\int_0^\infty e^{-x(2+it_1+t_2)}\,\frac{x^{(m-1)}}{(m-1)!}\,dx,\ee
and
\eq\label{2.22}f^{(m)}(x)=(-1)^m \int_\R e^{-x(1+it)} (1+it)^m \hf(t)\,dt\ee
obtained by integration by parts, one may write
\[S_W(f,\ga)\,=\,(\frac{\phi(W)}{W}\,log\,R)^{-m} \int_0^\infty (f^{(m)}(x))^2\,\frac{x^{(m-1)}}{(m-1)!}\,dx\, \times (1+o_{\om\to\infty}(1)+O_W((log\,R)^{-1}).\]
This shows \eqref{2.13}.
\\\\
To show \eqref{2.14} we modify the above argument as follows. We have
\eq\label{2.24}
S_{W,q}(f,\ga)\,=\,\int_\R\int_\R \sideset{}{'}\sum_{(D,W)=1} g_D(t_1,t_2)\,\ga_{[D,q]}(\vv)\,\hf(t_1)\hf(t_2)\,dt_1 dt_2.
\ee

For the inner sum we separate the cases $q\nmid D$ and $q|D$ in which case we change variables $D:=qD$, this gives
\[
\sideset{}{'}\sum_{(D,W)=1} g_D(t_1,t_2)\,\ga_{[D,q]}(\vv)\,=\, \ga_q(\vv)(1+g_q(t_1,t_2))\,\sideset{}{'}\sum_{\substack{(D,W)=1\\q\nmid D}} g_D(t_1,t_2)\,\ga_{[D,q]}(\vv)\,=\]
\[=\ \frac{\ga_q(\vv)(1+g_q(t_1,t_2))}{1+g_q(t_1,t_2)\ga_q(\vv)}\,\prod_{\substack{p>\om\\p\neq q}} (1-\frac{\ga_p(\vv)}{p^\frac{1+it_1}{log\,R}}-\frac{\ga_p(\vv)}{p^\frac{1+it_2}{log\,R}}+\frac{\ga_p(\vv)}
{p^\frac{2+it_1+it_2}{log\,R}}).
\]

Note that this differs from the integrand in\eqref{2.20} only by that additional factor
\eq\label{2.25}
\frac{\ga_q(\vv)(1+g_q(t_1,t_2))}{1+g_q(t_1,t_2)\ga_q(\vv)}\,=\,=\frac{m}{q} (1-q^{-\frac{1+it_1}{log\,R}})
(1-q^{-\frac{1+it_2}{log\,R}})\,(1+o_{\om\to\infty}(1)),
\ee
as by our assumption $q>\om$ hence $\ga_q(\vv)=\frac{m}{q}(1+o_{\om\to\infty}(1)$. Thus we have the analogue of \eqref{2.20}

\eq\label{2.26}
S_{W,q}(f,\ga)\,=\,\frac{m}{q}\,(\frac{\phi(W)}{W}\,log\,R)^{-m} \int_\R\int_\R \frac{(1+it_1)^m (1+it_2)^m}{(2+it_1+it_2)^m}\ \times
\ee
\[\,(1-e^{-\frac{(1+it_1)log\,q}{log\,R}})(1-e^{-\frac{(1+it_2)log\,q}{log\,R}})\,
\hf(t_1)\hf(t_2)\,dt_1dt_2\, \times (1+o_{\om\to\infty}(1)+O_W((log\,R)^{-1})
\]
\\
Finally, using \eqref{2.21} and \eqref{2.22} the one may rewrite the integral in \eqref{2.26} as

\eq\label{2.27}
\int_0^\infty \left(\int_R (e^{-x(1+it)}-e^{-(x+\frac{log\,q}{log\,R})(1+it)})\, (1+it)^m\,\hf(t)\,dt\right)^2 \frac{x^{m-1}}{(m-1)!}\ dx\ee
\[=\,\int_0^\infty \left(\int_R f^{(m)}(x) - f^{(m)}(x+\frac{log\,q}{log\,R})\right)^2\, \frac{x^{m-1}}{(m-1)!}\ dx.
\]
\end{proof}
\medskip

We turn to the proof of our main results now. First we prove Theorem \ref{thm1.2} which follows from Lemma 2.2. and Lemma 2.3 by routine calculation using the properties of the local factors $\si_p(p,\bb,\vv)$.

\bigskip

\emph{\underline{Proof of Theorem \ref{thm1.2}.}} Let $\eta\leq \eta(r,k)$, with $\eta(r,k)$ be as specified in \eqref{1.2.1}. By \eqref{2.6} and \eqref{2.13.1}, one has

\eq\label{2.38}
\sum_{\substack{\x\in[N]^n\\ (\LL(\x),W)=1,\ \F(\x)=\vv}} \La_R^2(l_1(\x)\cdots l_m(\x))\ =\
\sum_{\substack{\bb\in\Z_W^n\\ (\LL(\bb),W)=1}} S_{W,\bb}(N)\ =
\ee
\begin{multline*}\quad\quad\quad = c_m(f)\,N^{n-dr} J(N^{-d}\vv)\ \frac{W^{m-n}}{\phi(W)^m\,(log\,R)^m}\ \sum_{\substack{\bb\in\Z_W^n\\ (\LL(\bb),W)=1}}\ \Si_{W,\bb}(\vv)\,\times\\ \times\,(1+o_{\om\to\infty}(1))\  +\ O_W((log\,R)^{-1}).
\end{multline*}

By Lemma \ref{lem2.1} and the Chinese Remainder Theorem,
\eq\label{2.39}
\frac{W^{m-n}}{\phi(W)^m}\, \sum_{\substack{\bb\in\Z_W^n\\ (\LL(\bb),W)=1}}\Si_{W,\bb}(\vv)=\prod_{p|W} \frac{p^{m-n}}{(p-1)^m}\,\sum_{\substack{\bb\in\Z_p^n\\ (\LL(\bb),p)=1}}\si_p(p,\bb;\vv)\,)\ \prod_{p\nmid W}\si_p(\vv).
\ee
Note that $\si_p(\vv)=1+O(p^{-2})$ and hence $\prod_{p\nmid W} \si_p(\vv)=1+o_{\om\to\infty}(1)$. For a fixed $l\in \N$ and a prime $p\leq \om$, writing $\y:=p\,\x+\bb$ one has by \eqref{1.3.4}
\eq\label{2.40}
\frac{p^{m-n}}{(p-1)^m}\,p^{-l(n-r)}\sum_{(\LL(\bb),p)=1} |\{\x\in\Z_{p^l}^n;\ \F(p\x+\bb)\equiv\vv\ (mod\ p^l\}|
\ee
\[= \frac{p^m}{(p-1)^m}\,p^{-l(n-r)} |\{\y\in\Z_{p^l}^n;\ (\LL(\y),p)=1,\ \F(\y)\equiv\vv\ (mod\ p^l)\}|.\]
\\
Taking the limit $l\to \infty$, and recalling definitions \eqref{1.3.4} and \eqref{0.1}, we get
\[\frac{p^m}{(p-1)^m}\, \sum_{(\LL(\bb),p)=1} \si_p(p,\bb;\vv) = \si_p^*(\vv).\]
and then by \eqref{2.39}
\[\frac{W^{m-n}}{\phi(W)^m}\,\sum_{\substack{(\bb,W)=1\\ \F(\bb)\equiv \vv\ (mod\ W)}} \Si_{W,\bb}(\vv) = \prod_p \si_p^*(\vv)\ (1+o_{\om\to\infty}(1)).\]
This proves \eqref{1.2.1}.
\\\\
To prove \eqref{1.2.2} note that to estimate a sum over $\x\in [N]^n$ for which $\LL(\x)\notin (\PP^\eps(N))^m$ under the restriction $(\LL(\x),W)=1$, one needs to sum only over those $\x=(x_1,\ldots,x_n)$ for which $q|l_1(\x)\ldots l_m(\x)$ for some prime $\om<q\leq N^\eps$. Thus, recalling the definition of the sums $S_{W,\bb}(N,q)$ given in \eqref{2.5} we have that

\eq
\sum_{\substack{\LL(\x)\notin (\PP^\eps(N))^m\\(\x,W)=1,\,\F(\x)=\vv}} \La_R^2(l_1(\x)\cdots l_m(\x))\phi_N(\x)\,\leq\, \sum_{\om<q\leq N^\eps} S_{W,\bb}(N,q).
\ee
Let us make the simple observation that $|f^{(m)}(x)-f^{(m)}(x+\tau)|\leq \tau\,|f^{(m+1)}(x)|$ for $0\leq x,\tau\leq 1$ for our choice $f(x)=(1-|x|)_+^{8m}$.
Then by estimates \eqref{2.7} and \eqref{2.14}

$$
\sum_{\om <q\leq N^\eps} S_{W,\bb}(N,q)\leq \frac{m}{q}\left(\frac{log\,q}{log\,R}\right)^2 c_{m+1}(f)\,N^{n-rk}(log\,R)^{-m}\Si^*(N,\vv)(1+o_{\om\to\infty}(1))+O((log\,R)^{-2m}).
$$
\\
Write $\eps':=\eps/\eta$, so that $N^\eps = R^{\eps'}$ with $R=N^\eta$. The sum over the primes $\om<q\leq R^{\eps'}$ can be estimated by a dyadic decomposition using the Prime Number Theorem

$$
\sum_{\om< q\leq R^{\eps'}} q^{-1}\,(log\,q)^2\,=\,\sum_{\om \leq 2^j\leq\,R}\  \sum_{2^{j-1}<q\leq 2^j} q^{-1}\,(log\,q)^2\leq\,(2+o_{\om\to\infty}(1))\,\sum_{j\leq \eps' log_2\,R}\, j\,\leq 2\,(\eps')^2.
$$
This implies \eqref{1.2.2}.
\quad\quad\quad$\Box$
\bigskip

\underline{\emph{Proof of Theorem} \eqref{thm1.1}.} We need to choose $\eps>0$
to ensure that the expression in \eqref{1.2.2} is essentially less then the one in \eqref{1.2.1}. For that one needs to compare the quantities $c'_{m+1}(f)$ and $c_m(f)$ defined in Theorem \ref{thm1.2}. For our choice $f(x)=(1-|x|)_+^{8m}$ we have that $f^{(m)}(x)=\al_m (1-|x|)_+^{7m}$ while $f^{(m+1)}(x)=7m\al_m (1-|x|)_+^{7m}$ (with $\al_m=(8m)!/(7m)!$) for $0<x<1$, thus by the beta function identity
\[\int_0^1 (1-x)^a x^b\,dx=\frac{a!\,b!}{(a+b+1)!}\]
it is easy to see that $c'_{m+1}(f)< 32m^2\,c_m(f)$. Thus if $64m^3\, (\eps/\eta)^2 \leq 1/2$ then for $N$ and $\om=\om (\F,\LL)$ sufficiently large,

\eq\label{2.40}
\sum_{\substack{\LL(\x)\in (\PP^\eps(N))^m\\ \F(\x)=\vv}}
\La_R^2(\l_1(\x)\cdots l_m(\x))\ \phi_N(\x)
\geq\,c_m\,N^{n-kr}\,(log\,R)^{-m} \,\Si^*(N,\vv)\,
\ee
for some positive constant $c_m=c_m(f)>0$.
\\\\
Finally note that if $\x\in [N]^n$ and $\LL(\x)\in (\PP^\eps(N))^m$ then $l_i(\x)$ can have at most $1/\eps$ prime divisors, for every $1\leq i\leq m$, hence $\La_R(l_1(\x)\cdots l_m(\x))\leq 2^{m/\eps}$. Thus by \eqref{2.40} the number of solutions $\x\in [N]^n$ to $\F(\x)=\vv$ for which $\LL(\x)\in (\PP^\eps(N))^m$ is at least
\[
c(m,k,r)\,N^{n-kr}\,(log\,N)^{-m} \,\Si^*(N,\vv),
\]
with $c(m,k,r):= c_m\, 2^{-2m/\eps}$ for some $\eps=\eps(m,k,r)>0$. In fact one may choose $\eps:=(16m)^{-3/2}\eta(r,k)$ with $\eta(r,k)=(8r^2(r+1)(r+2)k(k+1))^{-1}$ given in \eqref{1.2.1}.
This proves Theorem \ref{thm1.1}
\quad\quad$\Box$

\bigskip

\section{The local factors.}  In this section we study the Euler factors $\ga_p(\vv)$ and prove the asymptotic formula \eqref{1.4.2}. Recall

\[\si_p (p,\s,\vv)=\lim_{l\to\infty} \si_p^l(p,\s,\vv),\ \ \ \ \text{where}\]
\[\si_p^l(p,\s,\vv)= p^{-l(n-r)}\,|\{\x\in\Z_{p^l}^n;\ \F(p\,\x+\s)\equiv 0\ (mod\ p^l)\}|.\]
\\
Note that this factor is non-zero only if $\F(\s)\equiv \vv\ (mod\ p)$. We call a point $\s\in\Z_p^n$  \emph{non-singular} if the Jacobian $Jac_\F (\s)$ has full rank $(=r)$ over the finite field $\Z_p$. In this case it is easy to calculate factors $\si_p^l(p,\s,\vv)$ explicitly.

\begin{lemma}\label{3.2} Let $\s\in\Z_p^n$ be a non-singular solution to the equation $\F(\s)\equiv \vv\ (mod\ p)$. Then
\eq \si_p^l(p,\s,\vv)=p^r.\ee
\end{lemma}

\begin{proof} We proceed by induction on $l$. For $l=1$ we have $\F(p\x+\s)\equiv \F(\s)\equiv \vv\ (mod\ p)$ for all $\x\in\Z_p^n$ thus $\si_p^1(p,\s,\vv)=p^r$. Let $l=2$. We'd like to count $\x\in\Z_{p^2}^n$ satisfying
\[\F(p\x+\s)\equiv \F(\s)+p\ Jac_\F(\s)\cdot\x\ \equiv \vv\ \ (mod\ p^2).\]
Since $\F(\s)-\vv=p\uu$ this reduces to
\[Jac_\F(\s)\cdot\x\equiv -\uu\ \ (mod\ p).\]
By assumption the map $Jac_\F(\s):\Z_p^n\to\Z_p^n$ has full rank, thus the above equation has $p^{n-r}$ solution in $\Z_p^n$ and hence $p^{2n-r}$ solutions $\x\in\Z_{p^2}^n$. It follows that $\si_p^2(p,\s,\vv)=p^r$.
For $l\geq 3$ we show that
\[\si_p^l(p,\s,\vv)=\si_p^{l-1}(p,\s,\vv).\]
Note that if $\x\equiv\y\ (mod\ p^{l-1})$ then $\F(p\x+\s)\equiv \F(p\y+\s)\ (mod\ p^l)$. For given $\y\in\Z_{p^{l-1}}^n$ write $\y=p^{l-2}\uu+\z$ with $\z\in\Z_{p^{l-2}}^n$ and $\uu\in\Z_p^n$. Then
\eq\label{3.2.1}
\F(p\y+\s)\equiv \F(p^{l-1}\uu+p\z+\s)\equiv \F(p\z+\s)+p^{l-1}Jac_\F(\s)\cdot\uu\ \ (mod\ p^l).
\ee
Thus $\F(p\y+\s)\equiv\vv\ (mod\ p^l)$ implies that
\eq\label{3.2.2}\F(p\z+\s)\equiv\ (mod\ p^{l-1}),\ee
the number such $\z\in\Z_{p^{l-2}}^n$ is $\ p^{-n}p^{(l-1)(n-r)}\si_p^{l-1}(p,\s,\vv)$. For a given $\z$ satisfying \eqref{3.2.2} write $\F(p\z+\s)=p^{l-1}\bb+\vv$, then \eqref{3.2.1} holds if and only if
\eq\label{3.2.3}
Jac_\F(\s)\cdot\uu \equiv -\bb\ \ \ \ (mod\ p).
\ee
By our assumption $Jac_\F(\s)$ has full rank $(=r)$ above $\Z_p^n$ thus the number of solutions to \eqref{3.2.3} is $p^{n-r}$. Since that decomposition $\y=p^{l-2}\uu+\z$ is unique it follows that
\[\si_p^l(p,\s,\vv)=p^n p^{-l(n-r)}p^{-n}p^{(l-1)(n-r)}p^{n-r}\si_p^{l-1}(p,\s,\vv)=\si_p^{l-1}(p,\s,\vv).\]
\end{proof}

For singular values of $\s$ we can only get upper bounds on the local factors $\si_p(p,\s,\vv)$. The case $\s=\vv=\b0$ suggests that one cannot get better estimates then $p^{kr}$.

\begin{lemma}\label{3.3} Let $\F$ be a family of $r$ integral forms of degree $k$, and assume that
\eq\label{codim}
codim\ (V_\F^*)\geq r(r+1)(k-1)2^k+1.
\ee
Then uniformly for $l\in\N$ and $\s\in\Z_p^n$ one has
\eq\label{3.3.1}
\si_p^l(p,\s,\vv)\ls p^{r^2k}.
\ee
\end{lemma}

\begin{proof} By \eqref{4.30} we have
\[\si_p^l(p,\s,\vv)=\sum_{t=0}^l \sum_{\bb\in\Z_{p^t}^r}^\ast p^{-tn} e^{-2\pi i\frac{\bb\cdot\F(p\x+\s)}{p^t}}\,S_{\bb,p^t}(p,\s),\]
where the sum in $\bb$ are taken over $r$-tuples with at least one coordinate not divisible by $p$, and $S_{\bb,p^t}(p,\s)$ is the exponential sum defined in \eqref{4.8}.
If $t>rk$ then Lemma \ref{lemma4.4} applies with $\eps=1/r$ (and $K=codim\ (V_\F^*)/2^{k-1}$)
thus
\[\sum_{t>rk}\ \sum_{\bb\in\Z_{p^t}^r}^\ast p^{-tn} |S_{\bb,p^t}(p,\s)|\,\ls \sum_{t>rk} p^{tr} p^{-\frac{tK}{(r+1)(k-1)}+\tau}\,\ls \sum_{t>rk} p^{-t\tau/2}\,\ls 1,\]
if $\tau=\tau(r,k)>0$ is chosen sufficiently small. Indeed, by \eqref{codim} we have $\frac{K}{(r+1)(k-1)}-r=\tau_0(r,k)>0$ and then \eqref{3.3.1} then follows from the trivial estimate $p^{-mn}|S_{\bb,p^m}(p,\s)|\leq 1$.
\end{proof}

\bigskip

\emph{Proof of Proposition \ref{prop1.2}.} Since $\si_p^{-1}(\vv)=1+O(p^{-2})$ for sufficiently large primes $p>\om$, it is enough to show that \eqref{1.4.2} holds for

\[\si_p(\vv)\ga_p(\vv):=p^{-n}\sum_{\F(\s)\equiv \vv\ (mod\ p)} \1_{p|l_1(\s)\cdots l_m(\s)} \si_p(p,\s,\vv)\]
\[= p^{-n}\sum_{\substack{\F(\s)\equiv 0\ (mod\ p)\\ \s\ non-singular}} \1_{p|l_1(\s)\cdots l_m(\s)} \si_p(p,\s,\vv)+ p^{-n}\sum_{\substack{\F(\s)\equiv 0\ (mod\ p)\\ \s\ singular}} \1_{p|l_1(\s)\cdots l_m(\s)} \si_p(p,\s,\vv)\]
\[=p^{-n+r}(\sum_{\F(\s)\equiv 0} \1_{p|l_1(\s)\cdots l_m(\s)} -\sum_{\substack{\F(\s)\equiv 0\\ \s\ singular}} \1_{p|l_1(\s)\cdots l_m(\s)}) +\ p^{-n}\sum_{\substack{\F(\s)\equiv 0\\ \s\ singular}} \1_{p|l_1(\s)\cdots l_m(\s)} \si_p(p,\s,\vv)\]
\[=: \ga_p^1(\vv) -\ga_p^2(\vv) +\ga_p^3(\vv).\]
\\
Let $V^*_\F(p)$ denote the locus of singular points $\s\in\Z_p^n$ of the $(mod\ p)$-reduced variety\\ $V_\F (p):=\{\F(\s)=\vv\}$. It is well-known fact in arithmetic geometry, see \cite{Shm}, that
\[codim\,(V^*_\F(p))=codim\,(V^*_\F),\]
for all but finitely many primes $p$, i.e. that codimension of the singular variety does not change when the equations defining the variety are considered $mod\ p$. Also, the number of points over $\Z_p$ on a homogeneous algebraic set $V$ is bounded by its degree times $p^{\,dim\,V}$, see \cite{GL} Prop. 12.1, hence
$|V^*_\F(p)|\ls p^{n-codim\,(V_\F^*)}$, where the implicit constant may depend on $n,k$ and $r$.
Thus for sufficiently large primes $p\geq \om$ we may apply Lemma \ref{3.2} which gives for $i=2,3$

\[|\ga_p^i(\vv)|\,\ls\, p^{-n+r^2 k} p^{n-codim\,(V_\F^*)}\, \ls\,  p^{r^2k-r(r+1)(k-1)2^{k-1}-1}\,\ls \,p^{-2}.\]
\\
For $1\leq i\leq m$ let define the subspace $M_i:=\{\s\in\Z_p^n;\ l_i(\s)=0\}$.
By the inclusion-exclusion principle we have that
\eq\label{3.8}
p^{-n+r}(\sum_{1\leq i\leq n}\sum_{\s\in M_i} \1_{\F(\s)=\vv}-\sum_{1\leq i<j\leq n}\, \sum_{\s\in M_i\cap M_j}\1_{\F(\s)=\vv})\, \leq\,\ga_p^1(\vv)\,\leq\, p^{-n+r}\sum_{1\leq i\leq n} \sum_{\s\in M_i} \1_{\F(\s)=\vv}.
\ee
\\
If $\F$ is a system of $r$ forms then it is easy to see that $rank(\F|_{M_J})\geq rank(\F)-r|J|$ for any subspace $M_J$ of codimension $|J|$, see \cite{CM1}, Cor. 2. For $J\subs [1,n]$ let $M_J:=\cap_{j\in J} M_j$, then by the pairwise linear independence of the forms $l_i$ we have that $codim\,M_J=|J|$ for $|J|\leq 2$, for sufficiently large $p$. By our assumption on the rank of the system $\F$ we have that for $1\leq |J|\leq 2$, $k\geq 2$
\[
rank(\F|_{M_J})-r \geq r(r+1)(k-1)2^k-r(|J|+1)\geq 2^{k-1}.
\]
Then by Proposition 4 in \cite{CM2} applied the the system $\F$ restricted to the subspace $M_J\simeq \Z_p^{n-|J|}$ one has
\eq\label{3.9}
p^{-(n-|J|)+r}\ \sum_{\s\in M_J} \1_{\F(\s)=\vv}\,=\,1+O\left(p^{-\frac{rank(\F|_{M_J})-r}{2^{k-1}}}\right)=1+O(p^{-1})
\ee
\\
This implies that the sums in \eqref{3.8} over the subspaces $M_i\cap M_j$
contribute $O(p^{-2})$ to the expression $\ga_p^1(\vv)$, while the sums over the subspaces $M_i$ are $p^{-1}+O(p^{-2})$.
This proves the Proposition. $\ \ \ \ \ \ \ \ \ \ \Box$

\bigskip

\section{Appendix: Diophantine equations over $\Z$ and $\Z_p$.} In this section we sketch the proof of Proposition \ref{prop1.1}, which is an extension of the main result of \cite{Bi}, see Theorem 1 there. Let us note that constants $\de,\eta,\eps$ etc. calculated here are different than the constants used in Section 2, the constants $\eta'$ and $\de'$ appear in Proposition \eqref{prop1.1} and all other constants in Section 2 are derived form those.
\\\\
For a family of integral forms $\F=(F_1,\ldots,F_r)$ and given $d\in\N$, $\s\in\Z^r$, $\al\in\R^r$ define the exponential sum
\eq\label{4.1}
S_N(d,\s,\bal):=\sum_{\x\in\Z^n} e^{2\pi i \al\cdot \F(d\x+\s)}\phi_N(d\x+\s),
\ee
where $\phi_N$ is the indicator function of a cube $B_N$ of size $N$.

Then one has the analogue of Lemma 2.1 in \cite{Bi}

\begin{lemma}\label{4.1} Let $1\leq d<N$, $N_1:=N/d$ and let $\s\in\Z^n$. Then
\eq\label{4.2}
|N_1^{-n} S_N(d,\s,\bal)|^{2^{k-1}}\ls N_1^{-kn}\ \sum_{\h^1,\ldots,\h^{k-1}\in [-N_1,N_1]^n} \prod_{j=1}^n \min\{N_1,\|d^k\bal\,\Phi_j( \Phi_j(\h^1,\ldots,\h^{k-1})\|^{-1}\},
\ee
where for $1\leq i\leq r$ the $i$=th component of the the multi-linear form $\Phi_j$ is given by
\[\Phi_j^i(\h^1,\ldots,\h^{k-1})=k!\,\sum_{1\leq j_1,\ldots,j_{k-1}\leq n} a^i_{j_1,\ldots,j_{k-1},j}\,h^1_{j_1},\ldots,h^{k-1}_{j_{k-1}},\]
and $\|\beta\|$ denotes the distance of a real number $\beta$ to the closest integer.
\end{lemma}

\begin{proof} Write
\eq\label{4.3}\F_{d,\s}(\x):=\F(d\x+\s)=d^k \F(\x)+G_{d,\s}(\x),\ \ \ \ \ deg\,( G_{d,\s})<k.\ee
Also, $\phi_N(d\x+\s)=\phi_{N_1,\s}(\x)$ where $\phi_{N_1,\s}$ is the indicator function of the cube $B_{N_1,\s}=d^{-1}(B_N-\s)$ of size $N_1$.

Introducing the differencing operators
\[D_\h F(\x):= F(\x+\h)-\F(\x),\]
as well as their multiplicative analogues
\[\De_\h \phi (\x):=\phi(\x+\h)\bar{\phi}(\x),\]
we have by applying the Cauchy-Schwarz inequality $k-1$-times
\eq\label{4.4}|
N_1^{-n}S_N(d,\s,\bal)|^{2^{k-1}}|\ \ls\ N_1^{-kn} \sum_{\h^1,\ldots,\h^{k-1}\in \Z^n} \left|\,\sum_{\x\in\Z^n} e^{2\pi i\bal\cdot D_{\h_{k-1}}\ldots D_{\h_1} \F_{d,\s}(\x)} \De_{\h_{k-1}}\ldots \De_{\h_1} \phi_{N_1,\s} (\x)\right|.
\ee

By \eqref{4.3} and \eqref{??} we have that
\[D_{\h_{k-1}}\ldots D_{h_1} \F_{d,\s}(\x)=d^k D_{\h_{k-1}}\ldots D_{h_1} \F(\x)=d^k\,\sum_{j=1}^n x_j\,\Phi_j(\h_1,\ldots,\h_{k-1}).\]

Estimate \eqref{4.2} then follows form the fact that $|\sum_{x\in I} e^{2pi \beta x}|\leq \min\{N_1,\|\beta\|^{-1}\}$ for any $\beta\in\R$, when the summation is taken over an interval $I$ of length at most $N_1$.
\end{proof}

\bigskip

Once this is established, the rest of the arguments in \cite{Bi} carry over to our situation leading the  following \emph{minor arcs} estimate. For given $1\leq d<N$, $N_1:=N/d$ and $0<\te<1$ define the system of \emph{major arcs}

\eq\label{4.5}
\MM(\te):=\bigcup_{1\leq q\leq N_1^{(k-1)r\te}} \bigcup_{(\baa,q)=1} \MM_{\baa,q}(\te),\ \ \ \ \textit{where}\ee
\[\MM_{\baa,q}(\te):= \{\bal\in [0,1]^r;\ |\al_i-a_i/q|\leq q^{-1} N_1^{-k+(k-1)r\te},\ 1\leq i\leq r\}.
\]

\medskip

\begin{lemma}\label{lemma4.2} [\cite{Bi}, Lemma 3.3] If $\{d^k\bal\}\notin \MM(\te)$ then one has for every $\tau>0$

\eq\label{4.6}
|S_N(d,\s,\bal)|\leq C_\tau\ N_1^{n-K\te+\tau}.
\ee
\end{lemma}

We need the above estimate in slightly different form, depending only on $\bal$.

\begin{lemma}\label{lemma4.3} Let $0<\te,\eps<1$ and let $0<\eta \leq \eps r(1-k^{-1})\te$. If $d\leq N^\frac{\eta}{1+\eta}$ then for $\bal\notin \MM(\te)$ one has uniformly for $\s\in \Z^n$

\eq\label{4.7}
|S_N(d,\s,\bal)|\ls_\tau\ N_1^{n-\frac{K}{1+\eps}\te +\tau}\ \ \ \ \ (\forall\ \tau>0).
\ee
\end{lemma}

\begin{proof} If $d^k\bal\in \MM_{\baa,q}(\te)$ (mod\ 1), then there is $q\leq N_1^{r(k-1)\te}$ and $a_i\in \Z$ such that $(a_i,q)=1$ and $|d^k\al_i-a_i/q|\leq q^{-1} N_1^{-k+(k-1)r\te}$. This implies that $|\al_i-a'_i/q_1|\leq q_1^{-1} N_1^{-k+(k-1)r\te}$ for some $q_1\leq d^k N_1^{(k-1)r\te}$ and $a_i'\in\Z$ for which $(a'_i,q_1)=1$. If $d\leq N^{\frac{\eta}{1+\eta}}$ then $d\leq N_1^\eta$ and hence $q_1\leq N_1^{k\eta+r(k-1)\te}\leq N_1^{(1+\eps)r(k-1)\te}$. This implies that $\al\notin \MM((1+\eps)\te$. By taking the contrapositive and changing variables $\te:=(1+\eps)\te$ the Lemma follows.
\end{proof}

As a first application we give an estimate for the Gauss sums

\eq\label{4.8}
S_{\baa,q}(d,\s):=\sum_{\x\in\Z_q^n} e^{2\pi i\frac{\baa\cdot\F(d\x+\s)}{q}}.
\ee

\begin{lemma}\label{lemma4.4} Let $q\in\N$ and $1\leq d < q^{\frac{\eps}{k}}$. Then for any $\baa\in \Z^r$ such that $(\baa,q)=1$ and $\s\in\Z^d$ one has
\eq\label{4.9}
|S_{\baa,q}(d,\s)|\ls_\tau q^{n-\frac{K}{(1+\eps)r(k-1)}+\tau}\ \ \ \ \ (\forall\ \tau>0).
\ee
\end{lemma}

\begin{proof} Note that $S_{\baa,q}(d,\s)=S_N(d,\s,\baa/q)$ with $N=dq$, as $\x\in [0,q)^n$ if $d\x+\s\in B_N=[0,dq)^n+\s$. Moreover  if $r(k-1)\te<1$ then for any $1\leq q'\leq q^{r(k-1)\te}<q$ and $(\baa',q')=1$
\[\left|\frac{\baa}{q}-\frac{\baa'}{q'}\right|\geq \frac{1}{qq'}> \frac{1}{q'}\  q^{-k+r(k-1)\te}.\]

This implies that $\baa/q\notin \MM(\te)$. Since $d < q^{\frac{\eps}{k}}$ we can choose $\te$ so that $r(k-1)\te<1$ but $d<q^\eta$ for $\eta:=\eps r (1-k^{-1})\eta$. The \eqref{4.7} implies that
\[|S_{\baa,q}(d,\s)|\ls_\tau q^{n-\frac{K}{(1+\eps)}\te +\tau} \ls_\tau q^{n-\frac{K}{(1+\eps)r(k-1)}+\tau}\ \ \ \ \ (\forall\ \tau>0),\]
choosing $\te$ sufficiently close to $\frac{1}{r(k-1)}$.
\end{proof}
$\ $\\
Taking $d=1$ and letting $\eps\to 0$ in \eqref{4.9} one has
\[S_{\baa,q}(1,\s)=S_{\baa,q}(1,0) \ls_\tau\ q^{n-\frac{K}{r(k-1)}+\tau}.\]
\\
Next, to apply \cite{Bi}, Lemma 4.4 adapted to our situation, we make the assumption that
\eq\label{4.10}
K>(1+\eps)r(r+1)(k-1).\ee
Then one can chose small positive numbers $\de$ and $\te_0$ so that
\eq\label{4.11}
\de+2r(r+2)\te_0\,<1\ee
and
\eq\label{4.12}
2\de\te_0^{-1}\,<\,K(1+\eps)^{-1}-r(r+1)(k-1).\ee

\begin{lemma}\label{lemma4.5} Let $\de,\te_0$ satisfy \eqref{4.10}-\eqref{4.11}, and let $0<\eta\leq\eps (1-k^{-1})\te_0$. Then for $1\leq d\leq N^\frac{\eta}{1+\eta}$ and $\s\in\Z^n$ one has
\eq\label{4.13}
\int_{\al\notin \MM(\te_0)} |S_N(d,\s,\al)|\,d\al \ls\ N_1^{n-dr-\de}.
\ee
\end{lemma}

If in addition we make the assumption that
\eq\label{4.14}
\eta<\de k^{-1}r^{-1},
\ee
then it is easy to see that
\eq\label{4.15}
N_1^{n-dr-\de}=N^{n-dr}d^{-n}\, N^{-\de}d^{kr+\de}\leq N^{n-dr}d^{-n}\, N^{-\de+(kr+\de)\eta (1+\eta)^{-1}}\leq N^{n-dr-\de'}d^{-n},
\ee
for some $\de'>0$.
\\\\
Going back to Proposition 1.2, we have under the conditions of Lemma \ref{4.5} and \eqref{4.14}

\[\RR_N(d,\s,\vv)=\int e^{-2\pi i\,\al\cdot\vv}\ S_N(d,\s;\al)\ d\al=\int_{\MM'(\te_0)} e^{-2\pi i\,\al\cdot\vv}\ S_N(d,\s;\al)\ d\al + O(N^{n-rd-\de'}d^{-n}),\]
for any set $\MM'(\te_0)\supseteq \MM(\te_0)$. From now on we will write
\eq\label{4.16} r(k-1)\te_0=\kappa,\ee
and define

\eq\label{4.17}
\MM'(\te_0):=\bigcup_{1\leq q\leq N_1^\kappa}
\bigcup_{(\baa,q)=1}\MM'_{\baa,q}(\te_0),\ \ \textit{where}\ee
\eq\label{4.18}
\MM'_{\baa,q}(\te_0):=\{\al\in [0,1]^r;\ |\al_i-a_i/q|\leq N_1^{-k+\kappa},\ 1\leq i\leq r\}.
\ee
\\
Next, for given $\al\in \MM'_{\baa,q}(\te_0)$, writing $\al=\baa/q+\be$ one has the following approximation of the sum $S_N(d,\s;\al$ (see \cite{Bi}, Lemma 5.1).

\begin{lemma}\label{4.6}
Let $0<\eta\leq \frac{1}{2},\ d\leq N^\frac{\eta}{1+\eta},\ \s\in\Z^n$. Then for $\al\in\MM'_{\baa,q}(\te_0)$
\eq\label{4.19}
S_N(d,\s;\al)=N^n d^{-n} q^{-n}S_{\baa,q}(d,\s)\,I(N^k\beta)\,+\,O(N^{n-1+2\eta+\kappa}d^{-n}),
\ee
where
\eq\label{4.20}
I(\ga):= \int_{\R^r} e^{2\pi i\ga\cdot \F(\y)} \phi(\y)\,d\y.
\ee
\end{lemma}

\begin{proof} Writing $\x:=q\y+\z$ with $\z\in [0,q)^n$, we have
\eq\label{4.21}
S_N(d,\s;\al)=\sum_{\z\in\Z_q^n} e^{2\pi i \frac{\baa\cdot \F(d\z+\s)}{q}} \sum_{\y\in\Z^n} e^{2\pi i \be\cdot \F(qd\y+d\z+\s)} \phi_N (qd\y+d\z+\s).
\ee
As $\y$ varies by $O(1)$ in the range $|qd\y|\ls N$, the variation in the exponent is
\[O(|\be|N^{k-1} qd)=O(N^{-1+2\kappa+\eta}).\]
Thus the error in replacing the sum $\sum e^{2\pi i \be\cdot \F(qd\y+d\z+\s)} \phi_N (qd\y+d\z+\s)$ by the integral
\[\int_{\y\in\R^r} e^{2\pi i \be\cdot \F(qd\y+d\z+\s)} \phi_N (qd\y+d\z+\s)\,d\y,\]
is $O(N^{n-1+2\kappa+\eta})$ + $O((N/dq)^{n-1})$. By a change of variables $\y:=N^{-1}(qd\y+d\z+s)$ we have
\[\int_{\y\in\R^r} e^{2\pi i \be\cdot \F(qd\y+d\z+\s)} \phi_N (qd\y+d\z+\s)\,d\y =
N^{n}d^{-n}q^{-n} I(N^k\be).\]
Summing over $\z\in\Z_q^n$, using \eqref{4.21} and \eqref{4.8} proves \eqref{4.19}.
\end{proof}

For $\mu\in\R^r$ and $\Phi>0$, write
\[J(\mu;\Phi):= \int_{|\ga_i|\leq \Phi} I(\ga)\,e^{-2\pi i \ga\cdot\mu}\,d\ga,\]
and define
\eq\label{4.22}
J(\mu):= \lim_{\Phi\to\infty} J(\mu;\Phi).
\ee

By Lemma 5.2 and Lemma 5.3 in \cite{Bi}, $J(\mu)$ exists, continuous and uniformly bounded by
\[\int_{\R^r} |I(\ga)|\,d\ga <\infty.\]

Also using assumption \eqref{4.9} and estimate \eqref{4.10} we have that the so-called \emph{singular series}
\eq\label{4.23}
\Si(d,\s;\vv):=\sum_{q=1}^\infty \sum_{(\baa,q)=1} q^{-n} e^{-2\pi i\,\frac{\baa\cdot\vv}{q}} S_{\baa,q}(d,\s)
\ee
is absolute convergent. In fact,

\eq\label{4.24}
\sum_{q\geq N_1^\kappa} \sum_{(\baa,q)=1} q^{-n} |S_{\baa,q}(d,\s)|\ls N^{-\de}.
\ee
\\
Indeed, as $\kappa=r(k-1)\te_0$ we have by assumption \eqref{4.10}
\[\frac{2\de}{\kappa}<\frac{K}{(1+\eps)r(k-1)} - r-1.\]
\\
Then by estimate \eqref{4.9}

\[\sum_{q\geq N_1^\kappa} \sum_{(\baa,q)=1} q^{-n} |S_{\baa,q}(d,\s)|
\ls_\tau \sum_{q\geq N_1^\kappa} q^{-\frac{2\de}{\kappa}+\tau} \ls_\tau N_1^{-2\de+\tau} \ls_\tau N^{-2\de+\de\eta+\tau}\ls N^{-\de}.\]

Summarizing we have

\begin{prop}\label{prop4.1} Let $\F=(F_1,\ldots,F_r)$ be a family of integral forms of degree $k$ satisfying the rank condition
\eq\label{4.25}
K:=\frac{codim\,V_\F^*}{2^{k-1}} > r(r+1)(k-1).
\ee
\\
There exists a constant $\de'=\de'(k,r)>0$ such that the following holds.\\

(i) If $0<\eta \leq \frac{1}{4r^2(r+1)(r+2)k^2}$ then for every $1\leq d\leq N^\frac{\eta}{1+\eta}$ and $\s\in\Z^n$ one has the asymptotic

\eq\label{4.27}
\RR_N(d,\s;\vv)=N^{n-rk}d^{-n} \Si(d,\s,\vv)\,J(N^{-k}\vv)\,+\,O(N^{n-rk-\de'}d^{-n}).
\ee
\\
(ii) Moreover if
\eq\label{4.28}
K > 2r(r+1)(k-1)+2rk,
\ee
then the asymptotic formula \eqref{4.27} holds for $\eta\leq \frac{1}{4r(r+2)k}$.
\end{prop}

\begin{proof} First we show that
if $0<\eps\leq 1$ satisfies $\eps<\frac{K}{r(r+1)(k-1)}-1$ and $\eta>0$ is such that
\eq\label{4.26}
\eta < \frac{1}{4r(r+2)k}\ \min\,\left\{\eps,\ \frac{K-(1+\eps)r(r+1)(k-1)}{r k(1+\eps)}\right\},
\ee
then \eqref{4.27} holds for $1\leq d\leq N^\frac{\eta}{1+\eta}$ and $\s\in\Z^n$.
\\

Set the parameters $\te_0$ and $\de$ as
\[\te_0:=\frac{1}{2r(r+2)k+1}\ ,\ \  \de:=\frac{\te_0}{2}\,\min\left\{1,\ \frac{K}{1+\eps}-r(r+1)(k-1)\right\},
\]
to satisfy conditions \eqref{4.11} and \eqref{4.12}. Then for
\[\eta< \frac{1}{4r(r+2)k}\ \min\left\{\eps,\ \frac{K-(1+\eps)r(r+1)(k-1)}{rk(1+\eps)}\right\},\]
we have that $\eta<\eps(1-k^{-1}\te_0$ and $\eta<\te_0 k^{-1}{r^{-1}}$ hence both the conditions of Lemma \ref{4.5} and \eqref{4.14} are fulfilled.

Thus by \eqref{4.16}, \eqref{4.19}, \eqref{4.24} and using the fact that $|\MM'(\te_0)|\leq N_1^{(r+1)\kappa -rk+r\eta} \leq N^{-rk+\frac{2}{3}}$ (by our choice of $\te_0,\ \eta$), we have that

\begin{eqnarray}
\RR_N(d,\s,\vv) &=& \sum_{q\leq N_1^\kappa} \sum_\baa \int_{\MM'_{\baa,q}(\te_0)} e^{-2\pi i\,\al\cdot\vv} S_N(d,\s;\al)\ d\al\ +\ O(N^{n-rk-\de'} d^{-n})\nonumber\\
&=& N^n d^{-n} \left(\ \sum_{q\leq N_1^\kappa} q^{-n} e^{-2\pi i\,\frac{\baa\cdot\vv}{q}} S_{\baa,q}(d,\s) \int_{|\be_i|\leq N_1^{-k+\kappa}} e^{-2\pi i \be\cdot\vv} I(N^k\be)\,d\be\ +\ O(N^{-rk-\frac{1}{3}+2\kappa+\eta})\right) \nonumber\\
&=& N^{n-rk} d^{-n} \left(\ \sum_{q\leq N_1^\kappa} q^{-n} e^{-2\pi i\,\frac{\baa\cdot\vv}{q}} S_{\baa,q}(d,\s) J(N^{-k}\vv;d^kN_1^\kappa)+O(N^{-\de'})\right)\nonumber\\
&=& N^{n-rk} d^{-n}\ \Si(d,\s;\vv) J(N^{-k}\vv)\ +\ O(N^{n-rk-\de'} d^{-n}),
\end{eqnarray}
for some $\de'=\de'(r,k)>0$.
\\\\
Indeed, the first line is \eqref{4.16}, the second line follows from \eqref{4.19} and the above remark on the size of the major arcs, the third line by a scaling $\be:=N^k\be$ and the last line from \eqref{4.?} and \eqref{4.24} together with the estimate $2\kappa+\eta\leq \frac{k-1}{k(r+2)}+\frac{1}{4r(r+2)k}\leq \frac{1}{r+2}(1-\frac{3}{4k})<\frac{1}{3}$.

If $K>r(r+1)(k-1)$ then $\frac{K}{r(r+1)(k-1)}-1\geq \frac{1}{r(r+1)(k-1)}$ thus one may choose $\eps$ slightly larger than $\frac{1}{r(r+1)k}$. This gives $\eta\leq \frac{1}{4r^2(r+1)(r+2)k^2}$ by \eqref{4.26}. If $K>2r(r+1)(k-1)+2rk$ then one may choose $\eps$ slightly larger than 1, which gives $\eta\leq \frac{1}{4r(r+2)k}$. This proves the Proposition.
\end{proof}

Finally, consider the singular series
\eq\label{4.29}
\Si(d,\s,\vv) = \sum_{q=1}^\infty q^{-n} \sum_{(\baa,q)=1} e^{-2\pi i \frac{\baa\cdot\vv}{q}} S_{\baa,q} (d,\s),
\ee
where writing $\F_{d,\s}(\x):=\F(d\x+\s)$
\[S_{\baa,q} (d,\s)=\sum_{\x\in\Z_q^n} e^{2\pi i\frac{\F_{d,\s}(\x)\cdot a}{q}}.\]
By the well-known multiplicative properties of the inner sums in \eqref{4.29}
\[\Si(d,\s,\vv)=\prod_{p\ prime} \si_p(d,\s,\vv),\]
with local factors
\[\si_p(d,\s,\vv)=\sum_{m=0}^\infty p^{-mn} \sum_{(\baa,p^m)=1} e^{-2\pi i \frac{\baa\cdot\vv}{p^m}} S_{\baa,p^m} (d,\s).\]
By estimate \eqref{4.9} and assumption \eqref{4.10} we have $\si_p(d,\s,\vv)=1+O(p^{-1-\de'})$ and hence the product is absolutely and uniformly convergent. Finally, by a straightforward calculation we have
\eq\label{4.30}
\si_p^l(d,\s;\vv):=\sum_{m=0}^l p^{-mn} \sum_{(\baa,p^m)=1} e^{-2\p i\frac{\baa\cdot\vv}{p^m}} S_{\baa,p^m}(d,\s)=p^{-l(n-r)}\,|\{\x\in\Z_{p^l}^n;\ \F_{d,\s}(\x)=\vv\}|.
\ee
Proposition 1.1 follows immediately from Proposition 4.1.

\end{document}